\documentclass{jimo}
\usepackage{amsmath}
  \usepackage{paralist}
  \usepackage{graphics} 
  \usepackage{epsfig} 
\usepackage{graphicx}  \usepackage{epstopdf}
\usepackage[colorlinks=true,urlcolor=blue, citecolor=red]{hyperref}

\usepackage{float}
\usepackage{latexsym}
\usepackage{enumitem}
\usepackage{amsthm}
\usepackage{amssymb}
\usepackage{graphicx}
\usepackage{array}\usepackage{multirow}\usepackage{subfigure}\usepackage[misc, geometry]{ifsym}
\usepackage[ruled, linesnumbered]{algorithm2e}
\def\R{\mathbb{R}}
\usepackage{subfigure}
\def\currentvolume{X}

    \def\ppages{X--XX}
     \def\DOI{xx.xxxx/xx.xx.xx.xx}

\providecommand{\tabularnewline}{\\}

\newtheorem{theorem}{Theorem}[section]
\newtheorem{corollary}{Corollary}[section]

\newtheorem{lemma}[theorem]{Lemma}
\newtheorem{proposition}{Proposition}[section]

\theoremstyle{definition}

\newtheorem{remark}{Remark}[section]

\newtheorem{example}{Example}[section]

\DeclareMathAlphabet{\mathcal}{OMS}{cmsy}{m}{n}
\DeclareSymbolFont{epsilon}{OML}{ntxmi}{m}{it}
\DeclareMathSymbol{\epsilon}{\mathord}{epsilon}{"0F}
\newcommand{\abs}[1]{\left|#1\right|}

\AtBeginDocument{
  
}

\title[Solving Strictly Quasiconvex Multiobjective Programming Problems] 
      {A Monotonic Optimization Approach for Solving Strictly Quasiconvex
Multiobjective Programming Problems}

\author[Tran Ngoc Thang and Dao Tuan Anh]{}


 \email{thang.tranngoc@hust.edu.vn}
 \email{dta.hust@gmail.com}

\thanks{}

\begin{document}
\maketitle

\centerline{\scshape Tran Ngoc Thang${}^1\footnotetext[1]{Corresponding author: thang.tranngoc@hust.edu.vn}$ and Dao Tuan Anh}
\medskip
{\footnotesize
 \centerline{ School of Applied Mathematics and Informatics}
   \centerline{Hanoi University of Science and Technology}
   \centerline{No. 1 Dai Co Viet, Hai Ba Trung, Hanoi, Vietnam}
}

\bigskip


\begin{abstract}
In this article, we use the monotonic optimization approach to propose
an outcome-space outer approximation by copolyblocks for solving strictly
quasiconvex multiobjective programming problems and especially in
the case that the objective functions are nonlinear fractional. After
the algorithm is terminated, with any given tolerance, we obtain an
approximation of the weakly efficient solution set, that contains
the whole weakly efficient solution set of the problem. The algorithm
is proved to be convergent and it is suitable to be implemented in parallel
using standard convex programming tools. Some computational experiments
are reported to show the accuracy and efficiency of the proposed algorithm.
\end{abstract}

\section{Introduction}

We consider the following strictly quasiconvex multiobjective programming problem 
\begin{align}
{\rm Vmin}\; & \;f(x)\tag*{\ensuremath{(QVP)}}\label{QVP}\\
\mbox{s.t.}\; & \;x\in\mathcal{S},\nonumber 
\end{align}
where the feasible solution set $\mathcal{S}\subset\R^{n}$, $n\in \mathbb{N}^*$ is a nonempty,
convex, compact set and the objective function $f:\R^{n}\rightarrow\R^{p}$, $2 \leq p\in \mathbb{N}^*$
is a strictly quasiconvex vector function on $\mathcal{S}$, i.e.
$f_{i},i=1,\dots,p$ are strictly quasiconvex functions on $\mathcal{S}$.
Recall that a continuous function $h:\mathcal{S}\rightarrow\R$ is
called \emph{strictly quasiconvex} if 
\[
h(x^{1})<h(x^{2})\Rightarrow h(\lambda x^{1}+(1-\lambda)x^{2})<h(x^{2}),
\]
for every $x^{1},x^{2}\in\mathcal{S}$ and $0<\lambda<1$ (see \cite{Avriel1998},
\cite{Mangasarian1969}). For two vectors $a,b\in\R^{r}$ with some
integer $r\geq2$, we denote $a\leq b$ if $a_{i}\leq b_{i}$ for
all $i=1,\dots,r$. We also write $a<b$ when $a_{i}<b_{i}$ for all
$i=1,\dots,r$. For any $a,b\in\R^{r}$ with $a\leq b$, the box $[a,b]$
is defined by the set of all $z\in\R^{r}$ such that $a\leq z\leq b.$
A feasible solution $\bar{x}$ is said to be an efficient solution
(resp., weakly efficient solution) of \ref{QVP} if there is
no solution $x\in\mathcal{S}$ such that $f(x)\leq f(\bar{x})$ and
$f(\bar{x})\not=f(x)$ (resp., $f(x)<f(\bar{x})$). 

In practical computation, finding the exact efficient solution set of problem \ref{QVP}
is very difficult, even impossible, even when \ref{QVP} is a linear
multiobjective programming problem \cite{Benson2005}. Therefore,
different methods to approximate the efficient solution set have been increasingly concerned (see \cite{Das1998}, \cite{Ehrgott2011}, \cite{GourionLuc2010},
\cite{Klamroth2002}, \cite{Lohne2014}, \cite{Ruzika2005},...).
Namely, given a vector $\varepsilon\in\R_{+}^{p}$, $\bar{x}$ is
said to be a weakly $\varepsilon$-efficient solution of $\ref{QVP}$
if there is no solution $x\in\mathcal{S}$ such that $f(\bar{x})-\varepsilon>f(x)$.
The sets of all efficient solutions, weakly efficient and weakly $\varepsilon$-efficient
solutions of $\ref{QVP}$ are respectively denoted by $\mathcal{S}_{E}$
and $\mathcal{S}_{WE}$ and $\mathcal{S}_{\varepsilon}$.

Denote the positive orthant of $\R^{p}$ by $\R_{+}^{p}$ and its
interior by ${\rm int}\R_{+}^{p}$. Let $Q\subset\mathbb{R}^{p}$
be a nonempty set. We denote by ${\rm Min}Q,{\rm WMin}Q$ and $Q_{\varepsilon}$
the sets of nondominated points, weakly nondominated points and weakly
$\varepsilon$-nondominated points of $Q$, respectively. Namely,
\begin{eqnarray*}
{\rm Min}Q & = & \{q^{0}\in Q\mid(q^{0}-\mathbb{R}_{+}^{p})\cap Q=\{q^{0}\}\},\\
{\rm WMin}Q & = & \{q^{0}\in Q\mid(q^{0}-{\rm int}\mathbb{R}_{+}^{p})\cap Q=\emptyset\},\\
Q_{\varepsilon} & = & \{q^{0}\in Q\mid(q^{0}-\varepsilon-{\rm int}\mathbb{R}_{+}^{p})\cap Q=\emptyset\}.
\end{eqnarray*}

\noindent We denote $\text{\ensuremath{\mathcal{Y}}}:=f(\mathcal{S})=\{y\in\R^p\mid\exists x\in \R^n,f(x)=y\}$
the outcome set or the value set of problem $\ref{QVP}$. With above
notations, the sets ${\rm Min}\mathcal{Y},{\rm WMin}\mathcal{Y}$
and $\mathcal{Y}_{\varepsilon}$ are the efficient, weakly efficient
and weakly $\varepsilon$-efficient outcome sets of $\ref{QVP}$,
respectively. They also are the images of $\mathcal{S}_{E},\mathcal{S}_{WE}$
and $\mathcal{S}_{\varepsilon}$ under $f$, respectively.

Recall that a vector function $f=(f_{1},\dots,f_{p}),f_{i}:\mathcal{S}\rightarrow\R$
is usually said to be \emph{convex} (resp., \emph{strictly quasiconvex}) if the
component functions $f_{i},i=1,\dots,p$ are convex (resp., strictly
quasiconvex) on $\mathcal{S}$ (see \cite{Benoist2001b}, \cite{Luc2005a}). It is easily seen that if $f$ is
convex then $f$ is strictly quasiconvex. Therefore, a \emph{convex
multiobjective programming problem} is just a special case of $\ref{QVP}$. 

As we know, there are many economic, financial or technical indicators
which are presented by ratios or fractional functions. The objective
functions, for instance, are maximization of output to input, return
to risk, profit to cost, or the rate of growth... (see \cite{Benson2005}, \cite{Konno1989},...). Consider two continuous functions $h,g$
on a nonempty convex set $S\subseteq\R^{n}.$ The fractional function
$h/g$ is strictly quasiconvex if $h$ is non-negative convex and
$g$ is positive concave on $S,$ or both $h$ and $g$ are affine
(for more other forms of the strictly quasiconvex fractional function,
see in \cite{Avriel1998}). By this assertion, we find that a \emph{multiobjective
concave fractional program} \cite{Benson2005} and a \emph{multiobjective
linear fractional program} \cite{Benson1998} are special cases of
$\ref{QVP}$. 

Several authors have studied the structure of the efficient value
set of $\ref{QVP}$. In this case, ${\rm Min}\mathcal{Y}$ is connected
(see \cite{Benoist2001b}, \cite{Luc2005a}) but is not closed even when $f$ is linear
fractional \cite{Stancu-Minasian97}. However, the weakly efficient
set ${\rm WMin}\mathcal{Y}$ is proved to be closed and connected.
Therefore, we establish outcome-space algorithm for approximating
the weakly efficient set ${\rm WMin}\mathcal{Y}$ instead of ${\rm Min}\mathcal{Y}$.
As usual, we consider the equivalently efficient set $\mathcal{Y}^{+}=\mathcal{Y}+\R_{+}^{p}$
which is full-dimensional and satisfies ${\rm WMin}\mathcal{Y}^{+}\cap\mathcal{Y}={\rm WMin}\mathcal{Y}.$
In general, $\mathcal{Y}^{+}$ is nonconvex, for example, $f(x)=\sqrt{\abs{x}},S=[-1,1]$, but $\mathcal{Y}^{+}$ has some nice property that it is a conormal
set. By the monotonic analysis, a conormal set can be approximated
by a copolyblock as closely as desired (see Section \ref{section:the_cutting_cones}). Therefore,
we propose an algorithm for outer approximating the set $\mathcal{Y}^{+}$
as well as the weakly efficient set ${\rm WMin}\mathcal{Y}^{+}.$
After the algorithm is terminated, we obtain an outer and an inner approximation
set of the weakly efficient value set (this idea is also employed in several previous works, e.g., \cite{Kaliszewski2018}.) An approximation
of the weakly efficient solution set is also obtained, which contains the whole weakly
efficient solution set $\mathcal{S}_{WE}$. The algorithm can be implemented
by using standard convex programming tools.

In Section 2, we present theoretical bases and algorithms to generate
a nondominated outcome point and a weakly efficient solution of problem
$(QVP).$ In this section, we also present The cutting cones and outer
approximate outcome sets to establish the outer approximation algorithm
for solving $(QVP)$ in Section 3. The convergence of the
algorithms are proved in Section 4, and Section 5 provides the computational
experiment. Some concluding remarks are given is the last section.

\section{Theoretical preliminaries}

\subsection{Generating a nondominated outcome point and a weakly efficient solution
of $(QVP)$}

Since the objective function $f$ is continuous and $\mathcal{S}$
is bounded, the outcome set $\mathcal{Y}$ is also bounded. Now we
determine a box containing $\mathcal{Y}$. 

By the compactness of $\mathcal{Y}$, for each $i=1,\dots,p,$ the
problems of minimizing and maximizing $y_{i}$ on the feasible set
$\mathcal{Y}$ have optimal solutions. It is easy to transform these
problems into 
\begin{equation}
\min\;\;\{f_{i}(x)\mid x\in\mathcal{S}\},\tag*{\ensuremath{({\rm P}_{i}^{m})}}\label{eq:P_i}
\end{equation}
and the problem 
\begin{equation}
\max\;\;\{f_{i}(x)\mid x\in\mathcal{S}\}.\tag*{\ensuremath{({\rm P}_{i}^{M})}}\label{eq:P_i-1}
\end{equation}
To solve problem \ref{eq:P_i}, we utilize the strictly quasiconvexity
of the objective function associated with the following remark.

\begin{remark}\label{rm_quasiconvex} \emph{Any local optimal solution
of a strictly quasiconvex programming problem is also a global optimal
solution \cite{Mangasarian1969}. Therefore, the problem can be solved
by using suitable algorithms for convex programming problems \cite{Benson2004}}.\end{remark}

By Remark \ref{rm_quasiconvex}, problem \ref{eq:P_i} can be solved
by convex programming tools. Note that \ref{eq:P_i-1} is a nonconvex
problem. However, it is possible to find an upper bound of this problem
without having to solve \ref{eq:P_i-1} (see \cite{Benson1998} for details and illustration).
Namely, for each $j=1,\dots,n$, set 
\[
\alpha_{j}=\min\{x_{j}\mid x\in\mathcal{S}\}.
\]
 Let $\alpha^{0}=(\alpha_{1},\alpha_{2},\dots,\alpha_{n})$ and 
\[
U=\max\{\left\langle e,x\right\rangle \mid x\in\mathcal{S}\},
\]
where $e$ is the vector of ones. Because $\mathcal{S}$ is a compact set, $U$ is a finite number. Notice also that convex programming
tools are applicable to find $\alpha^{0}$ and $U$. For each $j=1,2,\dots,n$,
define $\alpha^{j}=(\alpha_{1}^{j},\alpha_{2}^{j},\dots,\alpha_{n}^{j})^{T}$
by 
\[
\alpha_{k}^{j}=\begin{cases}
\alpha_{k}^{0},\; & {\rm if}\ \ k\ne j\\
U-\sum_{k\not=j}\alpha_{k}^{0},\; & {\rm if}\ \ k=j.
\end{cases}
\]
\begin{figure}[H]
\includegraphics[scale=0.65]{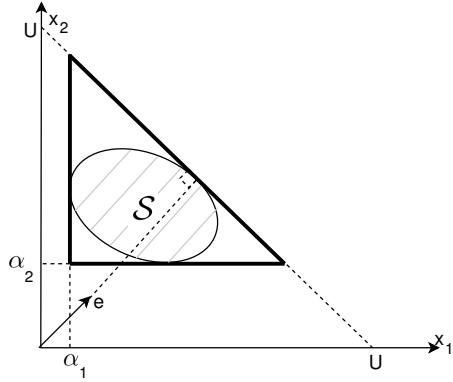}
\centering{}\caption{A 2D example of $\mathcal{S}$ and a simplex containing it}
\label{fig:simplex_example}
\end{figure}
Let $\Delta$ be a simplex with the vertex set $V(\Delta)=\{\alpha^{0},\alpha^{1},\dots,\alpha^{n}\}$.
It can be verified that $\mathcal{S}\subseteq\Delta$ (see Figure \ref{fig:simplex_example}). Therefore, 
\[
\max\{f_{i}(x)\mid x\in\mathcal{S}\}\leq\max\{f_{i}(x)\mid x\in\Delta\}.
\]
 Since $f_{i}(x),i=1,\dots,p,$ are quasiconvex and $\Delta$ is a
simplex, we have 
\[
\max\{f_{i}(x)\mid x\in\Delta\}=\max\{f_{i}(x)\mid x\in V(\Delta)\}.
\]
 For each $i=1,\dots,p$, choose a real number $M_i$ such that
\[
M_{i}=\max\{f_{i}(x)\mid x\in V(\Delta)\}.
\]
 Then 
\[
M_{i}\geq\max\{f_{i}(x)\mid x\in\mathcal{S}\}.
\]

Denote the optimal value of problem \ref{eq:P_i} by $m_{i}$, for
$i=1,\dots,p$. Let $m=(m_{1},m_{2},\dots,m_{p})$ and $M=(M_{1},M_{2},\dots,M_{p}).$
Then we get the box $[m,M]$ such that $\mathcal{Y}\subseteq[m,M].$
The point $m$ is also known as the ideal point of the outcome set.
If $m\in\mathcal{Y}$, the set ${\rm Min}\mathcal{Y}$ consists of
this point only. From now on, we assume that $m\not\in\mathcal{Y}$. 

Let
\begin{eqnarray*}
\mathcal{P}^{0} & = & [m,M]=(m+\Bbb R_{+}^{p})\cap(M-\Bbb R_{+}^{p});\\
\mathcal{Y}^{+} & = & \mathcal{Y}+\Bbb R_{+}^{p};\\
\mathcal{Y}^{\diamond} & = & \mathcal{Y}^{+}\cap(M-\Bbb R_{+}^{p}).
\end{eqnarray*}
It is clear that $\mathcal{Y}^{+}$ and $\mathcal{Y}^{\diamond}$
have interior points and $\mathcal{Y}^{\diamond}\subset\mathcal{P}^{0}$.
The following evident properties of $\mathcal{Y}^{+}$ and $\mathcal{Y}^{\diamond}$
will be used in the sequel (see \cite{Benson1998}). 

\begin{proposition} \label{lem_MinY} We have

i) $\;{\rm Min}\mathcal{Y}={\rm Min}\mathcal{Y}^{+}={\rm Min}\mathcal{Y}^{\diamond}$;

ii) ${\rm WMin}\mathcal{Y}={\rm WMin}\mathcal{Y}^{+}\cap\mathcal{Y}={\rm WMin}\mathcal{Y}^{\diamond}\cap\mathcal{Y}$.

\end{proposition}

Let $\hat{d}\in{\rm int}\R_{+}^{p}$, i.e. $\hat{d}>0$ be a fixed
vector and choose an arbitrary point $v\in\Bbb R^{p}$. We denote
$\ell(v)=\{v+t\hat{d}\mid t\in\Bbb R\}$ to be the line through $v$
with direction $\hat{d}$. The intersection of $\ell(v)$ and the
boundary of $\mathcal{Y}^{+}$ is determined by 
\begin{equation}
w_{v}=v+t_{v}\hat{d},\label{eq:bar_w}
\end{equation}
 where $t_{v}$ is the optimal solution of the problem 
\begin{equation}
\begin{array}{rl}
\min & t\tag*{\ensuremath{(P^{0}(v))}}\\
{\rm s.t.} & v+t\hat{d}\in\mathcal{Y}^{+},\;t\in\Bbb R.
\end{array}\label{eq:P_0_v}
\end{equation}
The following assertion shows the way to determine a weakly nondominated
outcome point.

\begin{lemma} \label{lem_bar_w} Let $v$ be an arbitrary point in $\R^{p}$.
Then there exists the unique point $w_{v}$ determined by \eqref{eq:bar_w}
and it is a weakly nondominated point of $\mathcal{Y}^{+}$.

\noindent \end{lemma}

\noindent \begin{proof} Due to the boundedness of $\mathcal{Y}$,
given an arbitrary $v\in\Bbb R^{p}$, there always exists a translation
of axes so that $v$ and $\mathcal{Y}$ are two proper subsets of
$\Bbb R_{+}^{p}$. Hence, without loss of generality, we can make
an assumption $v\cup\mathcal{Y}^{+}\subset{\rm int}\Bbb R_{+}^{p}$
. Under this assumption, two possible cases may occur, namely, $v\in\mathcal{Y}^{+}$
and $v\notin\mathcal{Y}^{+}$. We investigate the lemma in each case.

Firstly, if $v\notin\mathcal{Y}^{+}$, we denote $\ell^{+}=\{v+t\hat{d}\mid t\geq0\}$
to be the ray starting from $v$ along direction $\hat{d}$. As a
result of \cite{Benoist2001b}, $\ell^{+}$ and $\partial\mathcal{Y}^{+}$
always intersect at a unique point $w_{v}\in{\rm WMin}\mathcal{Y}^{+}$.

If $v\in\mathcal{Y}^{+}$, because $\mathcal{Y}^{+}\in{\rm int}\Bbb R_{+}^{p}$,
there does not exist any line which is a subset of $\mathcal{Y}^{+}$.
Due to the closedness of $\mathcal{Y}^{+}$, $\ell(v)\cap\mathcal{Y}^{+}$
is also closed. Let $T$ be the feasible domain of \ref{eq:P_0_v}
and $t^{*}$ its optimal solution. Due to a property of the projection
$\Pi:\ell(v)\rightarrow\Bbb R$, we also have that $T$ is a closed
set. Obviously, if $t\in T$ and $t'>t$ then $t'\in T$. By definition,
$\ell^{+}=\{v+t\hat{d}\mid t\geq0\}\in\mathcal{Y}^{+}$. Thus, either
$t^{*}=-\infty$ or $t^{*}\in T$ finite and $t^{*}\leq0$. If $t^{*}=-\infty$
then $\ell(v)$ is a proper subset of $\mathcal{Y}^{+}$. This statement
contradicts the fact that such line does not exist. This yields $t^{*}\leq0$
has to be a finite real number. At that point, we let $w_{v}=v+t^{*}\hat{d}$.
It can easily be seen that $w$ belongs to the boundary of $\mathcal{Y}^{+}$.
Indeed, by definition, we already have $\bar{w}\in\mathcal{Y}^{+}$.
Moreover, for all $\delta>0$, the ball $B_{\delta}(w)$ of radius
$\delta$ centered at $w$ always contains a point $\tilde{w}=v+\tilde{t}\hat{d}\in\Gamma,\tilde{t}<t^{*}$
which does not belong to $\mathcal{Y}^{+}$.

Assuming $w\notin{\rm WMin}\mathcal{Y}^{+}$, then there exists a
point $w'\in(w-{\rm int}\Bbb R_{+}^{p})\cap\mathcal{Y}^{+}<w$. Therefore,
the ball $B_{\delta'}(w)$ centered at $w$ of some positive radius
$\delta'$ such that $B_{\delta'}(w)\subset w'+{\rm int}\Bbb R_{+}^{p}\subset\mathcal{Y}^{+}$
exists, which contradicts $w\in\partial\text{\ensuremath{\mathcal{Y}}}^{+}$.
The proof is complete.\end{proof}

The explicit form of \ref{eq:P_0_v} is the following problem 
\begin{align}
\min\,\, & t\tag*{\ensuremath{(P^{1}(v))}}\label{eq:P_1_v}\\
\mbox{s.t.\,\,} & f(x)-t\hat{d}-v\leq0,\nonumber \\
 & x\in\mathcal{\mathcal{S}},\;t\in\Bbb R.\nonumber 
\end{align}
This problem is nonconvex in general, for example, $f(x)=\sqrt{|x|},\mathcal{\mathcal{S}}=[-1,1].$
Therefore, it is difficult to determine a weakly nondominated outcome
point as well as a weakly efficient solution of $(QVP)$. However,
we can transform problem \ref{eq:P_0_v} into the form 
\begin{align}
\min\,\, & \max\{\dfrac{f_{j}(x)-v_{j}}{\hat{d}_{j}}\mid j=1,...,p\}\tag*{\ensuremath{(P^{2}(v))}}\label{eq:Pbarv}\\
\mbox{s.t.}\;\; & x\in\mathcal{\mathcal{S}}.\nonumber 
\end{align}
It is worthy to note that \ref{eq:Pbarv} can be viewed as a weighted Chebyshev function (see \cite{Kaliszewski2012}) of which weights are $\frac{1}{\hat{d}_j}$. The following lemma shows that \ref{eq:Pbarv} is equivalent to \ref{eq:P_0_v}
and it is a problem of minimizing a strictly quasiconvex function
over a convex set. 

\begin{lemma}\label{lem_equiv} Problems \ref{eq:P_0_v} and \ref{eq:Pbarv}
are equivalent, i.e. if $(x^{*},t^{*})$ is the optimal solution of
\ref{eq:P_0_v} then $x^{*}$ is the optimal solution of \ref{eq:Pbarv};
conversely, if $x^{*}$ is the optimal solution and $t^{*}$ is the
optimal value of problem \ref{eq:Pbarv} then $(x^{*},t^{*})$ is
the optimal solution of \ref{eq:P_0_v}. Moreover, problem \ref{eq:Pbarv}
is a strictly quasiconvex programming problem.\end{lemma}

\begin{proof}We consider this proof under the known equivalence of
\ref{eq:P_0_v} and \ref{eq:P_1_v}. Let $(x^{*},t^{*})$ be the optimal
solution of \ref{eq:P_1_v}. The feasible condition of \ref{eq:P_1_v}
can be rewritten as
\[
t\geq\max\{\dfrac{f_{j}(x)-v_{j}}{\hat{d}_{j}}\mid j=1,\dots,p\},\;x\in\mathcal{S},\;t\in\Bbb R.
\]

Assuming there exists $x\in\mathcal{S}$ such that $$\max\{\dfrac{f_{j}(x)-v_{j}}{\hat{d}_{j}}\mid j=1,\dots,p\}<\max\{\dfrac{f_{j}(x^{*})-v_{j}}{\hat{d}_{j}}\mid j=1,\dots,p\}.$$Let $t=\max\{\dfrac{f_{j}(x)-v_{j}}{\hat{d}_{j}}\mid j=1,\dots,p\}$
then $(x,t)$ is feasible and corresponds to a better objective value
of \ref{eq:P_1_v}, which is contrary. Therefore, $$x^{*}=\min\{\max\{\dfrac{f_{j}(x)-v_{j}}{\hat{d}_{j}}\mid j=1,\dots,p\},x\in S\}.$$This is sufficient to conclude that $x^{*}$ is the optimal solution
of \ref{eq:Pbarv}.

On the other hand, let $x^{*}$ be the optimal solution and $t^{*}$
be the optimal value of \ref{eq:Pbarv}. We have $x^{*}\in\mathcal{S}$
and $t^{*}=\max\{\dfrac{f_{j}(x^{*})-v_{j}}{\hat{d}_{j}}\mid j=1,\dots,p\}\in\Bbb R$
so that $(x^{*},t^{*})$ belongs to the feasible domain of \ref{eq:P_1_v}.
We now assume there exists $x\in\mathcal{S},t\in\Bbb R$ such that
$t<t^{*}$ and $f(x)-t\hat{d}-v\leq0$. This yields that $t\geq\max\{\dfrac{f_{j}(x)-v_{j}}{\hat{d}_{j}}\mid j=1,\dots,p\}$.
Therefore, the objective value of $(P^{2}(f(x)))$ is less than $t^{*}$.
This contradicts the fact that $t^{*}$ is the optimal value of \ref{eq:Pbarv}.
In other words, $(x^{*},t^{*})$ must be the optimal solution of \ref{eq:P_1_v}.

Because each $f_{j}(x)$ is a strictly quasiconvex function and each
$\hat{d}_{j}>0$, from \cite{Mangasarian1969}, we have that $\max\{\dfrac{f_{j}(x)-v_{j}}{\hat{d}_{j}}\},j=1,\dots,p$
are also strictly quasiconvex functions. Hence, \ref{eq:Pbarv} is
a strictly quasiconvex programming problem.\end{proof}

The next theorem is crucial for the method to generate a nondominated
outcome point and a weakly efficient solution of $(QVP)$.

\begin{theorem} For any point $v\in\R^{p}$, let $x_{v}$ and $t_{v}$
be the optimal solution and the optimal value of the problem \ref{eq:Pbarv},
respectively. Then, $w_{v}=v+t_{v}\hat{d}$ is a weakly nondominated
point of $\mathcal{Y}^{+}$ and $x_{v}$ is a weakly efficient solution
of $(QVP)$. \end{theorem}

\begin{proof}According to Lemma \ref{lem_equiv}, $(x_{v},t_{v})$
is also the optimal solution of \ref{eq:P_0_v}. By Lemma \ref{lem_bar_w},
we conclude that $w_{v}\in{\rm WMin}\mathcal{Y}^{+}$. The feasible
domain of \ref{eq:P_1_v} suggests that $x_{v}$ satisfies $f(x_{v})\leq v+t_{v}\hat{d}=w_{v}$.
Thus, $f(x_{v})-{\rm int}\Bbb R_{+}^{p}\subseteq w_{v}-{\rm int}\Bbb R_{+}^{p}$
so that $(f(x_{v})-{\rm int}\Bbb R_{+}^{p})\cap\text{\ensuremath{\mathcal{Y}}}^{+}=\emptyset$.
Because $f(x_{v})\in{\rm WMin}\mathcal{Y}^{+}$, $x_{v}$ is therefore
a weakly efficient solution of the problem \ref{QVP}.\end{proof}

\begin{remark}\label{rm_convextools}From Lemma \ref{lem_equiv},
problem \ref{eq:Pbarv} is a strictly quasiconvex programming problem.
Therefore, by Remark \ref{rm_quasiconvex}, \ref{eq:Pbarv} can be
solved by using some algorithms for convex programming problems.\end{remark}

The following corollary presents the way to verify the weakly efficient
condition of any point in the decision space.

\begin{corollary} \label{cor:verify_we_solution}The point $x^{*}\in\R^{n}$
is a weakly efficient solution of $(QVP)$ if and only if the problem
$(P^{2}(f(x^{*}))$ has the optimal value $t^{*}=0.$ \end{corollary}

\begin{proof}Let $v^{*}=f(x^{*})\in\Bbb R^{p}$. If $t^{*}=0$ is
the optimal value of $(P^{2}(v^{*}))$, using the equivalence in Lemma
\ref{lem_equiv}, and from Lemma \ref{lem_bar_w}, we have $w^{*}=v^{*}+t^{*}\hat{d}\equiv v^{*}=f(x^{*})\in{\rm WMin}\mathcal{Y}^{+}$.
Hence, $x^{*}\in\mathcal{S}_{WE}$. Conversely, suppose that $x^{*}$
is a weakly efficient solution of \ref{QVP}. Then $v^{*}=f(x^{*})\in\partial\mathcal{Y}^{+}$.
It is obvious that for $t<0$ we have $v^{*}+t\hat{d}\notin\text{\ensuremath{\mathcal{Y}}}^{+}$
and $t^{*}=0$ is the smallest value of $t$ which satisfies the feasible
condition of \ref{eq:P_0_v}.\end{proof}

Procedure 1 shows the way to generate a weakly efficient solution
of \ref{QVP} from a point $v\in\R^{p}$.\medskip{}

\noindent \begin{algorithm}[H]
\SetAlgorithmName{Procedure 1}{procedure}{List of Procedures}
\renewcommand{\thealgocf}{}
\let\oldnl\nl
\newcommand{\nonl}{\renewcommand{\nl}{\let\nl\oldnl}}
\caption{\textit{GenerateWES(v)}}
\label{algo:generate}
\KwIn{A point $v\in \R^p$}
\KwOut{A nondominated outcome point and a weakly efficient solution of \ref{QVP} (Remark \ref{rm_convextools})}
\vspace{0.1cm}
Solve problem $({\rm P}^{2}(v))$ to find an optimal solution $(x_v,t_v)$\\
Set $w_v \gets v + t_v \hat{d}$\\
({\it $x_v$ is a weakly efficient solution of \ref{QVP} and $w_v$ is a weakly nondominated point of $\mathcal{Y}^+$})
\end{algorithm}

\medskip{}
Due to Corollary \ref{cor:verify_we_solution}, an arbitrary point
$x\in\R^{n}$ can be verified to be a weakly efficient solution of
\ref{QVP} by Procedure 2.

\noindent \begin{algorithm}[H]
\SetAlgorithmName{Procedure 2}{procedure}{List of Procedures}
\renewcommand{\thealgocf}{}
\let\oldnl\nl
\newcommand{\nonl}{\renewcommand{\nl}{\let\nl\oldnl}}
\caption{\textit{VerifyWES($x^*$)}}
\label{algo:verify}
\KwIn{A point $x^*\in \R^n$}
\vspace{0.1cm}
Set $v \gets f(x^*)$\\
Solve problem $({\rm P}^{2}(v))$ to find an optimal solution $(x_v,t_v)$\\
\If {$t_v = 0$} {
	$x^*$ is a weakly efficient solution of \ref{QVP}
}
\end{algorithm}

\subsection{The cutting cones and outer approximate outcome sets}
\label{section:the_cutting_cones}
For convenience, we first recall some concepts of monotonic optimization
which have been developed in \cite{Tuy1999} and \cite{Tuy2000}.
Consider a set $Q\subset\R^{p}$ contained in box $[m,M].$ The
set $Q$ is called \emph{normal} if $[m,y]\subset Q$ for all $y\in Q,$
and is called \emph{conormal} if $[y,M]\subset Q$ for all $y\in Q.$ Throughout
this paper, we only consider the concepts related to conormal sets.

It is known that the intersection of any family of conormal sets is
a conormal set. The intersection of all conormal sets containing $Q$
is called the \emph{conormal hull} of $Q$ and denoted by $\mathcal{N}(Q)$
. It is also the smallest conormal set containing $Q$.

The conormal hull of a finite set $V\subset[m,M]$ is said to be \emph{copolyblock}
$P,$ with vertex set $V,$ i.e. $P=\bigcup_{v\in V}[v,M]$ or $P=\mathcal{N}(V).$
A vertex $v\in P$ is called \emph{proper} if there is no vertex $v'\in P$
such that $v'\not=v$ and $v'\leq v.$ An \emph{improper} vertex of
$P$ is an element of $V$ which is not a proper vertex. Obviously,
a copolyblock is fully determined by its proper vertex set. It means
that a copolyblock is the conormal hull of its proper vertices.

The following propositions recall some main properties of copolyblocks
and will be used in the sequel.

\begin{proposition}\label{prop_copolyblock}

(i) The intersection of finitely many copolyblocks is a copolyblock.

(ii) The minimum of an increasing function over a copolyblock is achieved at a proper vertex of this copolyblock.

(iii) Any compact conormal set is the intersection of a family of
copolyblocks.

\end{proposition}

It is easily seen that the outcome set $\mathcal{\mathcal{Y}^{\diamond}}$
is a compact conormal set in the box $[m,M].$ Let a point $v\in[m,M]$
and $v\not\in\mathcal{Y}$. Then, by Lemma \ref{lem_bar_w}, the point
$w_{v}$ determined by \eqref{eq:bar_w} is a weakly nondominated
point of $\mathcal{Y}^{\diamond}$ and $w_{v}\in\partial\mathcal{\mathcal{Y}^{\diamond}}.$
Moreover, since $\hat{d}>0,$ we have $w_{v}>v.$ By Proposition 2.3
in \cite{TuyThach2005}, the cone $\mathcal{C}(w_{v}):=w_{v}-\R_{+}^{p}$
separates $v$ strictly from $\mathcal{Y}^{\diamond}.$ We refer to
the cone $\mathcal{C}(w_{v})$ as the \emph{cutting cone of} $\mathcal{Y}^{\diamond}$
\emph{at} $w_{v}.$ 

\begin{proposition}\label{prop_newV}

Let $P$ be a copolyblock in the box $[m,M]$ with proper vertex set
$V$ such that $\mathcal{\mathcal{Y}^{\diamond}}\subseteq P.$ For
a given $v\in[m,M]\setminus\mathcal{\mathcal{Y}^{\diamond}}$ and
$w_{v}$ determined by \eqref{eq:bar_w}, new copolyblock $P'$ obtained
by applying the cutting cone of $\mathcal{\mathcal{Y}^{\diamond}}$
at $w_{v}$ has vertex set $V',$ where 
\[
V'=\ensuremath{(V\setminus\{v\})\cup\{v-(w_{i}-v_{i})e^{i}\},\quad i=1,\dots,p.}
\]
\end{proposition}

By Proposition \ref{prop_copolyblock} (iii), we can approximate any
compact conormal set by a copolyblock as closely as desired, similarly
to approximating a compact convex set by a polytope. Therefore, the
compact conormal set $\mathcal{Y}^{\diamond}$ can be approximated
by a family of copolyblocks. Specifically, a nested sequence of copolyblocks
is generated that outer-approximates the outcome set $\mathcal{Y}^{\diamond},$
i.e.
\[
\mathcal{P}^{0}\supset\mathcal{P}^{1}\supset\dots\supset\mathcal{P}^{k}\supset\mathcal{P}^{k+1}\supset\dots\supset\mathcal{Y}^{\diamond},
\]
where the initial copolyblock $\mathcal{P}^{0}=[m,M]$ as constructed
in Section 2.1. The copolyblock $P^{k+1}$ is generated from $P^{k}$
by applying the cutting cone procedure in Proposition \ref{prop_newV}.
The copolyblocks $\mathcal{P}^{k}$ are said to be the \emph{approximate
outcome sets}. From these outer approximate sets, we shall establish
an outer approximation algorithm for solving $(QVP).$ 

\section{The algorithm for solving $(QVP)$ }

By the outcome space approach, the solution set of $(QVP)$ is achieved
by determining an approximation set contained in $\mathcal{Y}_{\varepsilon}^{\diamond}$.
Firstly, we find the set of weakly nondominated points $\mathcal{Y}_{WN}$
and the set of the outer approximation vertices $V_{\varepsilon}.$
Then we can determine the inner and outer approximation set $\mathcal{L}$
and $\mathcal{U}$ of $\mathcal{Y}^{\diamond}.$

\medskip

We propose an outer approximation algorithm to determine weakly nondominated
points of the outcome set $\mathcal{Y}$. Starting from the copolyblock
$\mathcal{P}^{0}=[m,M],$ we iteratively construct a sequence of copolyblocks
$\{\mathcal{P}^{k}\}$ such that 
\[
\mathcal{P}^{0}\supset\mathcal{P}^{1}\supset\dots\supset\mathcal{P}^{k}\supset\mathcal{P}^{k+1}\supset\dots\supset\mathcal{Y}^{\diamond}.
\]
 We will make use of the following notations: 
\begin{itemize}
\item $V^{k}$ is the set of all proper vertices which determines $\mathcal{P}^{k}$;
\item $V_{\varepsilon}$ is a collection of outer approximate weakly nondominated
points;
\item $\mathcal{Y}_{WN}$ is a collection of weakly efficient values. 
\end{itemize}
At the initial step with $k=0,$ we have $V^{0}=\{m\}$, $V_{\varepsilon}=\emptyset$
and $\mathcal{Y}_{WN}=\emptyset$.

\noindent In a typical iteration $k$, if every vertex in $V^{k}$
is an outer approximate weakly nondominated point, the algorithm terminates.
Otherwise, there is some $v^{k}\in V^{k}\setminus V_{\varepsilon}$.
In this case, by solving $({\rm P}^{2}(v^{k})),$ we find a new weakly
efficient value $f(x^{k})$ of \ref{QVP} and add it into the set
$\mathcal{Y}_{WN}.$ We also obtain a weakly nondominated point $w^{k}=v^{k}+t_{k}\hat{d}$
of $\mathcal{Y}^{+}$, where $(x^{k},t_{k})$ is the optimal solution
of $({\rm P}(v^{k}))$. If $v^{k}$ is close enough to $w^{k}$, $v^{k}$
is an outer approximate weakly nondominated point and is added to
$V_{\varepsilon}$. Otherwise, a new approximation $\mathcal{P}^{k+1}$
is determined by applying the cutting cone procedure in Proposition
\ref{prop_newV}. As proved later, for sufficiently large $k$, all
vertices of $\mathcal{P}^{k}$ are close enough to $\mathcal{Y}^{\diamond}$
and the algorithm terminates. This algorithm is described in Algorithm \textit{Solve(QVP)}
and Procedure \textit{RIV($V^{k+1}$, $v^k$, $w^k$)} as follows.

\medskip

\begin{algorithm}[H]
\SetAlgorithmName{Procedure \textit{RIV($V^{k+1}$, $v^k$, $w^k$)}}{procedure}{List of Procedures}
\renewcommand{\thealgocf}{}
\let\oldnl\nl
\newcommand{\nonl}{\renewcommand{\nl}{\let\nl\oldnl}}
\caption{}
\label{algo:riv}
\KwIn{The new vertex set $V^{k+1}$ probably including improper elements, previously chosen $v^k$, corresponding weakly nondominated point $w^k$.}
\KwOut{The new vertex set $V^{k+1}$ with all proper elements.}
\vspace{0.1cm}
\ForEach {$w \in V^k \setminus \{ v^k \}$} {
	\If {$w \geq v^k$ and $w_k < w^k_i$ for exactly one $i$ in $\{1,\dots,p\}$} {
		\tcc{i.e., $\exists i$ such that $w_i < w^k_i, w_j \geq w^k_j\;\; \forall j \neq i,\;\;i,j \in \{1,\dots,p\}$.}
		Remove $w_i^k$.
	}
}
\end{algorithm}

\noindent \begin{algorithm}[H]
\SetAlgorithmName{Algorithm \textit{Solve(QVP)}}{algorithm}{List of Procedures}
\renewcommand{\thealgocf}{}
\let\oldnl\nl
\newcommand{\nonl}{\renewcommand{\nl}{\let\nl\oldnl}}
\caption{}
\label{algo:solve_qvp}
\SetAlgoRefName{Solve(QVP)}
\KwIn{Objective function $f$ and constraint sets.}
\KwOut{The collection of outer approximate weakly nondominated points.}
\vspace{0.1cm}
Choose a tolerance level $\varepsilon = \epsilon{e}\geq 0$ where $\epsilon\in \R$ and $e=(1,\dots,1)\in\R^p$.\\
Find two points $m, M$ by solving the problems $(P_{i}^{m}), (P_{i}^{M})$ for $i=1,\dots,p$;\\
\nonl $\mathcal{P}^{0}\gets[m,M].$\\
Determine the set $V^0$ and choose a direction $\hat{d}>0$, for instance $\hat{d}=e$. \\
Initialize $V_{\varepsilon} \gets \emptyset; \; \mathcal{Y}_{WN} \gets \emptyset; \; k \gets 0.$\\
\While{$V^{k}\setminus V_{\varepsilon} \not= \emptyset\;$} {
	Choose any $v^{k}\in V^{k}\setminus V_{\varepsilon}$.\\
	Solve problem $({\rm P}^{2}(v^{k}))$ to find an optimal solution $(x^{k},t_{k})$ and set\\
	\nonl\quad $w^{k} \gets v^{k}+t_{k}\hat{d}$;\\
	\nonl\quad $\mathcal{Y}_{WN} \gets \mathcal{Y}_{WN}\cup\{f(x^{k})\}$.\\
	\uIf {$\Vert w^{k}-v^{k}\Vert\leq\epsilon$} {
		$V_{\varepsilon}\gets V_{\varepsilon}\cup\{v^{k}\}$;\\
		\textbf{continue}.\\
	}
	\Else {
		$V^{k+1}\gets (V^{k}\setminus\{v^{k}\})\cup\{v^{k}-(w_{i}^{k}-v_{i}^{k})e^{i}\},i=1,\dots,p$;\\
		Remove improper elements by Procedure \textit{RIV}($V^{k+1}$, $v^k$, $w^k$).
	}
	$k \gets k+1$;
}
\vspace{0.1cm}
\KwRet{$V_{\varepsilon}$}.
\end{algorithm}

\medskip

Suppose the algorithm is terminated at Iteration $K.$ Then, we obtain
two sets $\mathcal{Y}_{WN}$ and $V_{\varepsilon}.$ From these sets,
we define

\[
\mathcal{L}:=\mathcal{N}(\mathcal{Y}_{WN})
\]
and

\[
\mathcal{U}:=\mathcal{N}(V_{\varepsilon})\equiv\mathcal{P}^{K}.
\]
It can be verified that $\mathcal{L}$ and $\mathcal{U}$ are inner
and outer approximation of $\mathcal{Y}^{\diamond}$ respectively,
and their weakly nondominated sets are approximate weakly nondominated
sets of $\mathcal{Y}^{\diamond}.$ Based on the outer approximation $\mathcal{U},$ we can establish
the outer approximation $ES$ of the weakly efficient solution set
$\mathcal{S}_{WE}$ of problem $(QVP),$ that is 

\[
ES:=\underset{y\in\mathcal{U}_{\varepsilon}\cap\mathcal{Y}^{\diamond}}{\bigcup}\{x\in\mathcal{S}\mid f(x)\leq y\}
\]

By Corollary \ref{thm_ES} in the following section, it is proved
that $ES$ contains the weakly $\varepsilon-$efficient solution of
$(QVP)$ and also contains the whole weakly efficient solution set
of this problem.

\section{The convergence of Algorithm $Solve$$(QVP)$}

We will consider Hausdorff distance between two closed sets $Q_{1},Q_{2}\subset\Bbb R^{p}$
defined as follows.

\begin{eqnarray*}
d_{H}(Q_{1},Q_{2}) & = & \mbox{inf}\{t>0:Q_{1}\subseteq Q_{2}+tU_{p},Q_{2}\subseteq Q_{1}+tU_{p}\}\\
 & = & \mbox{max}\{\underset{v_{1}\in Q_{1}}{\mbox{sup}}d(v_{1},Q_{2}),\underset{v_{2}\in Q_{2}}{\mbox{sup}}d(v_{2},Q_{1})\},
\end{eqnarray*}
 where $U_{p}$ is the closed unit ball in $\Bbb R^{p}$ and the distance
from a point $v$ to a set $Q\subset\Bbb R^{p}$ is defined by $d(v,Q)=\underset{y\in Q}{\mbox{inf}}||v-y||$.
We say that a sequence of nonempty closed sets $\{Q_{k}\}_{k=1}^{\infty}\subseteq\Bbb R^{p}$
converges to a closed set $Q$ if $\mbox{lim}_{k\rightarrow\infty}d_{H}(Q_{k},Q)=0$
and write $\mbox{lim}_{k\rightarrow\infty}Q_{k}=Q$.

\begin{lemma}\label{lem:d_equal}

For each $v\in\mathcal{P}^{0}\setminus\mathcal{Y}^{\diamond}$, we
have $d(v,\mathcal{Y}^{\diamond})=d(v,\mathcal{Y}^{+})$.

\end{lemma}

\begin{proof}

We consider the distance between a point $v\in\mathcal{P}^{0}\setminus\mathcal{Y}^{\diamond}$
and $\mathcal{Y}^{+}$
\begin{eqnarray*}
d(v,\mathcal{Y}^{+}) & = & \min\{d(v,y)\mid y\in\mathcal{Y}^{+}\}\\
 & = & \min\{d(v,y)\mid y\leq f(x),x\in\mathcal{S}\}.
\end{eqnarray*}

Since $\mathcal{S}$ is compact, an optimal solution of the above
minimization problem always exists. In other words, the distance between
$v$ and $\mathcal{Y}^{+}$ is finite, and therefore there exists
a closed ball $V_{r}(v)$ of radius $r>0$ centered at $v$ satisfying
\begin{eqnarray}
\underset{y\in\mathcal{Y}^{+}}{\min}d(v,y) & = & \underset{y\in\mathcal{Y}^{+}\cap V_{r}(v)}{\min}d(v,y)\nonumber \\
\Leftrightarrow d(v,\mathcal{Y}^{+}) & = & d(v,\mathcal{Y}^{+}\cap V_{r}(v)).\label{eq:lemma_d_equal_1}
\end{eqnarray}

Let $Q$ be a compact conormal set which does not contain $v$, and
$z^{*}$ be the projection of $v$ onto $Q$. Then, there exists a
closed ball $V_{||z^{*}-v||}(v)$ centered at $v$ having $z^{*}$
a boundary point.

Suppose that $z^{*}\notin v+\Bbb R_{+}^{p}$, then there exists a
point $\tilde{z}\neq z^{*},\tilde{z}\notin z^{*}+\Bbb R_{+}^{p}$
satisfying $\tilde{z}\in\mbox{int}V_{||z^{*}-v||}(v)$. Since $Q$
is a conormal set, $\tilde{z}\in Q$. This contradicts the assumption
of $z^{*}$ because $d(v,\tilde{z})<d(v,z^{*})$. Thus,

\begin{equation}
z^{*}\in v+\Bbb R_{+}^{p}.\label{eq:lemma_d_equal_2}
\end{equation}

Now for each $y'\in(v+\Bbb R_{+}^{p})\cap Q$, let $V_{||v-y'||}(v)$
be the closed ball centered at $v$ having a boundary point $y'$.
Since $y'\in v+\Bbb R_{+}^{P}$, it is clear that $((y'+\Bbb R_{+}^{p})\setminus y')\cap V_{||v-y'||}(v)=\emptyset.$
Thus, $d(v,y)>d(v,y'),\,\forall y\in y'+\Bbb R_{+}^{p}.$ Therefore,
$d(v,y)$ is a continuous, increasing function of variable $y$ on
$(v+\Bbb R_{+}^{p})\cap Q$. From \cite{Tuy1999}, its global minimum
is achieved at a nondominated point of the domain. We now apply this
argument twice, with $Q$ replaced by $\mathcal{Y}^{\diamond}$ and
$\mathcal{Y}^{+}\cap V_{r}(v)$. It is worth noting from Proposition
\ref{lem_MinY} that ${\rm Min}\mathcal{Y}^{\diamond}\subseteq{\rm WMin}\mathcal{Y}^{\diamond}\cap\mathcal{Y}$
and ${\rm Min}\mathcal{Y}^{+}\subseteq{\rm WMin}\mathcal{Y}^{+}\cap\mathcal{Y}$.
Combining these facts with \eqref{eq:lemma_d_equal_1} and \eqref{eq:lemma_d_equal_2},
we deduce that the projection of $v$ on $\mathcal{Y}^{\diamond}$
must be the optimal solution of $\min d(v,y)$, subject to $y\in{\rm WMin}\mathcal{Y}^{\diamond}\cap\mathcal{Y}\cap(v+\Bbb R_{+}^{p})$.
Similarly, the projection of $v$ on $\mathcal{Y}^{+}$ must be the
optimal solution of $\min d(v,y)$, subject to $y\in{\rm WMin}\mathcal{Y}^{+}\cap\mathcal{Y}\cap(v+\Bbb R_{+}^{p})$.
It is followed by Proposition \ref{lem_MinY}(ii)
that the feasible domains of the two problems are exactly the same.
The proof is complete.\end{proof}

\begin{lemma}\label{lem:Bk_Ydiamond_space_bounded}

At the $k^{th}$ iteration of the algorithm, let $w_{v}$ be the weakly
nondominated point obtained by solving $(P^{2}(v))$ with some $v\in V^{k}$,
then

\[
d_{H}(\mathcal{P}^{k},\mathcal{Y}^{\diamond})\leq\underset{v\in V^{k}}{\max}||w_{v}-v||.
\]
\end{lemma}

\begin{proof}Since $V^{k}$ contains all vertices of $\mathcal{P}^{k}$,
it is obvious that

\[
\max\{d(v,\mathcal{Y}^{\diamond})\mid v\in\mathcal{P}^{k}\}=\underset{z\in V^{k}}{\max}\{\max\{d(v,\mathcal{Y}^{\diamond})\mid v\in(z+\R_{+}^{p})\cap(M-\R_{+}^{p})\}\}.
\]
Because $d(v,\mathcal{Y}^{\diamond})$ is a convex function and the
box $(z+\R_{+}^{p})\cap(M-\R_{+}^{p})$ is convex,
\[
\max\{d(v,\mathcal{Y}^{\diamond})\mid v\in(z+\R_{+}^{p})\cap(M-\R_{+}^{p})\}=d(z,\mathcal{Y}^{\diamond}).
\]
Therefore, 
\[
d_{H}(\mathcal{P}^{k},\mathcal{Y}^{\diamond})=\underset{v\in\mathcal{P}^{k}}{\max}d(v,\mathcal{Y}^{\diamond})=\underset{v\in V^{k}}{\max}d(v,\mathcal{Y}^{\diamond}).
\]
Since $w_{v}\in\mathcal{Y}^{+}$ for all $v\in V^{k}$ and from the
result of Lemma \ref{lem:d_equal}, we have 
\[
\underset{v\in V^{k}}{\max}d(v,\mathcal{Y}^{\diamond})=\underset{v\in V^{k}}{\max}d(v,\mathcal{Y}^{+})\leq\underset{v\in V^{k}}{\max}||w_{v}-v||.
\]
This completes the proof.\end{proof}

\begin{lemma}

\label{lem:d'}

For each $v\in\mathcal{P}^{0}\setminus\mathcal{Y}^{\diamond}$, there
exists a point $M'>M$ such that the weakly nondominated point $w_{v}$
of $\mathcal{Y}^{+}$ obtained by solving $(P^{2}(v))$ lies in box
$[m,M']$.\end{lemma}

\begin{proof}

Denote $\mathcal{N}_{d}(Q):=(Q+\R_{+}^{p})\cap(d-\R_{+}^{p})$ the
conormal hull of $Q$ in the box $[m,d]$ where $Q$ is some set contained
in the box $[m,M]$. Let us rewrite \ref{eq:P_1_v}, the equivalent
problem to \ref{eq:Pbarv}, in this form 
\[
\begin{array}{cc}
\min & t\\
{\rm s.t.} & f(x)\in\mathcal{N}_{v+t\hat{d}}(\mathcal{Y}^{+}),\\
 & x\in\mathcal{S},\;t\geq0.
\end{array}
\]

We denote $t_{m}$ and $t_{v}$ to be the optimal values of $(P^{2}(m))$
and $(P^{2}(v))$, respectively. The existence of these values has
been proved in Lemma \ref{lem_bar_w}. Since $m<v$ for all $v\in\mathcal{P}^{0}\setminus\mathcal{Y}^{\diamond}$,
$\{(x,t)\mid f(x)\in\mathcal{N}_{m+t\hat{d}}(\mathcal{Y}^{+}),x\in\mathcal{S},t\geq0\}$
the feasible domain of $(P^{2}(m))$ is a subset of $\{(x,t)\mid f(x)\in\mathcal{N}_{v+t\hat{d}}(\mathcal{Y}^{+}),x\in\mathcal{S},t\geq0\}$
the feasible domain of any $v\in\mathcal{P}^{0}\setminus\mathcal{Y}^{\diamond}$.
Therefore, the relation $t_{v}\leq t_{m}$ holds for such $v$. We
can therefore write $w_{v}=v+t_{v}\hat{d}\leq v+t_{m}\hat{d}\leq M+t_{m}\hat{d}=M'$.

Moreover, it is obvious that $m\leq v\leq w_{v}$, which completes
the proof.\end{proof}

\begin{lemma}\label{lem:lim_v_wv}For $\varepsilon=0$ either of two
following statements is true.

i) There exists a finite number $k$ such that 
\[
\underset{v\in V^{k}}{\max}||w_{v}-v||=0;
\]

ii) The number $k$ tends to infinity and 
\[
\underset{k\rightarrow\infty}{\lim}\underset{v\in V^{k}}{\max}||w_{v}-v||=0,
\]
where $V^{k}$ is the set of all proper vertices determining $\mathcal{P}^{k}$
and $w_{v}$ is the corresponding weakly nondominated point of $\mathcal{Y}^{+}$
obtained by solving $(P^{2}(v))$.

\end{lemma}

\begin{proof}Let $\hat{\mathcal{P}}^{0}=[m,M'],$ with $M'=M+t_{m}\hat{d}$.
At the $k^{th}$ iteration, we can also determine a copolyblock $\hat{\mathcal{P}}^{k+1}$
in box $[m,M']$ along with $\mathcal{P}^{k+1}$. It is clear that
$\mathcal{P}^{k}\subseteq\hat{\mathcal{P}}^{k}$ and $\hat{\mathcal{P}}^{k+1}\subseteq\hat{\mathcal{P}}^{k}$
for any $k\geq0$. From Lemma \ref{lem:d'}, $w_{v}\in\hat{\mathcal{P}}^{k}$,
for each $v\in\mathcal{P}^{k}\setminus\mathcal{Y}^{\diamond}$. Now
consider $v^{k}\in\mathcal{P}^{k}$ chosen at the $k^{th}$ iteration
and $t_{k}$ the optimal values of $(P^{2}(v^{k}))$. As before, let
$w_{v^{k}}=v^{k}+t_{k}\hat{d}$. We have
\begin{equation}
\mbox{Vol}[v^{k},w_{v^{k}}]=(t_{k})^{p}\mbox{Vol}[0,\hat{d}].\label{eq:lem_lim_v_wv_2}
\end{equation}
The lemma holds if $\underset{v\in V^{k}}{\max}||w_{v}-v||=0$ at
some $k\geq0$. Otherwise, there exists $v^{k}\in V^{k}$ such that
$||w_{v^{k}}-v^{k}||=\underset{v\in V^{k}}{\max}||w_{v}-v||>0$. We
also have $\mathcal{P}^{k+1}\subseteq\mathcal{P}^{k}\setminus(v^{k}-{\rm int}\Bbb R_{+}^{p})$,
noting that the equality holds when no improper vertex appears during
the cut. Since $[v^{k},w_{v^{k}}]\subseteq\hat{\mathcal{P}}^{k}$
followed by the definition of $w_{v^{k}}$, the volume of $\hat{\mathcal{P}}^{k}$
satisfies 
\begin{equation}
\mbox{Vol}\hat{\mathcal{P}}^{k}-\mbox{Vol}\hat{\mathcal{P}}^{k+1}\geq\mbox{Vol}[v^{k},w_{v^{k}}].\label{eq:lem_lim_v_wv_1}
\end{equation}
Combining \eqref{eq:lem_lim_v_wv_1} with \eqref{eq:lem_lim_v_wv_2},
we obtain 
\[
\mbox{Vol}\hat{\mathcal{P}}^{k}-\mbox{Vol}\hat{\mathcal{P}}^{k+1}\geq(t_{k})^{p}\mbox{Vol}[0,\hat{d}].
\]
Therefore, 
\[
\begin{array}{ccc}
\sum_{i=0}^{k}(\mbox{Vol}\hat{\mathcal{P}}^{i}-\mbox{Vol}\hat{\mathcal{P}}^{i+1}) & \geq & \left(\sum_{i=0}^{k}(t_{i})^{p}\right)\mbox{Vol}[0,\hat{d}]\end{array}.
\]
We deduce 
\[
\mbox{Vol}\hat{\mathcal{P}}^{0}\geq\left(\sum_{i=0}^{k}(t_{i})^{p}\right)\mbox{Vol}[0,\hat{d}]+\mbox{Vol}\hat{\mathcal{P}}^{k+1}\geq\left(\sum_{i=0}^{k}(t_{i})^{p}\right)\mbox{Vol}[0,\hat{d}],
\]
for all $k\geq1$. Thus, by letting $k\rightarrow\infty$, the positive
series $\sum_{i=0}^{\infty}(t_{i})^{p}$ is upper bounded by $\mbox{Vol}\hat{\mathcal{P}}^{0}/\mbox{Vol}[0,\hat{d}],$
and therefore converges. Since $||\hat{d}||$ is bounded, for any
$i\geq1$, we have

\[
\lim_{i\rightarrow\infty}\underset{v\in V^{i}}{\max}||w_{v}-v||=\lim_{i\rightarrow\infty}||w_{v^{i}}-v^{i}||=\lim_{i\rightarrow\infty}t_{i}||\hat{d}||=0.
\]
\end{proof}

\begin{lemma}\label{lem:ydiamond_convergence}With any $\varepsilon\geq0$,
the sets $\mathcal{P}^{k}$ with $k\geq0$ satisfy $\mathcal{Y}^{\diamond}\subseteq\mathcal{P}^{k+1}\subseteq\mathcal{P}^{k}.$
Moreover, by setting $\varepsilon=0$, we have 
\begin{equation}
\begin{aligned}\mathcal{Y}^{\diamond} & = & \lim_{k\rightarrow\infty}\mathcal{P}^{k} & = & \bigcap_{k\geq1}\mathcal{P}^{k},\\
{\rm WMin}\mathcal{Y}^{\diamond} & = & \lim_{k\rightarrow\infty}{\rm WMin}\mathcal{P}^{k}.
\end{aligned}
\label{eq:lem_ydiamond_convergence_0}
\end{equation}

\end{lemma}

\begin{proof}

The proof falls naturally into three parts. From the formulation of
$\mathcal{P}^{k},k\geq0$, the first part of the lemma is immediate.
We deduce directly from Lemma \ref{lem:Bk_Ydiamond_space_bounded}
and Lemma \ref{lem:lim_v_wv} that
\[
\lim_{k\rightarrow\infty}d_{H}(\mathcal{P}^{k},\mathcal{Y}^{\diamond})\leq\lim_{k\rightarrow\infty}\max_{v\in V^{k}}||w_{v}-v||=0.
\]
Therefore, $\{\mathcal{P}^{k}\}_{k\geq0}$ converges to $\mathcal{Y}^{\diamond}$
when $k$ goes to infinity.

Since $\mathcal{P}^{k+1}\subseteq\mathcal{P}^{k}$ for any $k\geq0$,
it follows easily that $\lim{}_{k\rightarrow\infty}\mathcal{P}^{k}=\bigcap_{k\geq1}\mathcal{P}^{k}$
. What is left is to prove the second equation of \eqref{eq:lem_ydiamond_convergence_0}.
Let $Q$ be a copolyblock in box $[m,M]$, and recall the notation
$Q^{+}=Q+\Bbb R_{+}^{p}$, we first prove 
\begin{equation}
{\rm WMin}Q=\partial Q^{+}\cap(M-\Bbb R_{+}^{p}).\label{eq:lem_ydiamond_convergence_1}
\end{equation}
If $\bar{w}\in{\rm WMin}Q$ then of course $\bar{w}\in(M-\Bbb R_{+}^{P})$.
Since the cone $\bar{w}-{\rm int}\Bbb R_{+}^{p}\nsubseteq Q^{+}$,
it is clear that $\bar{w}\in\partial Q^{+}$. This yields ${\rm WMin}Q\subseteq\partial Q^{+}\cap(M-\Bbb R_{+}^{p})$.
Conversely, if $\bar{w}\in\partial Q^{+}\cap(M-\Bbb R_{+}^{p})$.
By the property of a copolyblock, $\bar{w}-{\rm int}\Bbb R_{+}^{p}\nsubseteq Q^{+}$.
This also means that $\bar{w}$ is a weakly nondominated point of
$Q^{+}$. Because $\bar{w}<M$, we also have $\bar{w}\in{\rm WMin}Q$.
\eqref{eq:lem_ydiamond_convergence_1} is proved.\\
Applying \eqref{eq:lem_ydiamond_convergence_1} on $\mathcal{Y}^{\diamond}$
gives 
\begin{equation}
{\rm WMin}\mathcal{Y}^{\diamond}=\partial\mathcal{Y}^{+}\cap(M-\Bbb R_{+}^{p}).\label{eq:lem_ydiamond_convergence_3}
\end{equation}
Let $\partial\mathcal{P}^{k+}$ be the boundary set of $\mathcal{P}^{k+}:=\mathcal{P}^{k}+\R_{+}^{p}$.
We now apply \eqref{eq:lem_ydiamond_convergence_1} again, with $\mathcal{Y}^{\diamond}$
replaced by $\mathcal{P}^{k+}$, to obtain

\begin{equation}
{\rm WMin}\mathcal{P}^{k}=\partial\mathcal{P}^{k+}\cap(M-\Bbb R_{+}^{p})\subseteq\mathcal{P}^{k}.\label{eq:lem_ydiamond_convergence_2}
\end{equation}
Moreover, since $\mathcal{Y}^{+}\subseteq\mathcal{P}^{k+}$, 
\[
d_{H}(\partial\mathcal{P}^{k+}\cap(M-\Bbb R_{+}^{p}),\partial\mathcal{Y}^{+}\cap(M-\Bbb R_{+}^{p}))
\]
\[
=\max\{d(v,\partial\mathcal{Y}^{+}\cap(M-\Bbb R_{+}^{p}))\mid v\in\partial\mathcal{P}^{k+}\cap(M-\Bbb R_{+}^{p})\}.
\]
 From Lemma \ref{lem:d_equal}, for each $v$ inside the box $[m,M]$
which does not belong to $\mathcal{Y}^{\diamond}$, we have 
\[
d(v,\mathcal{Y}^{\diamond})=d(v,\mathcal{Y}^{+})=d(v,\partial\mathcal{Y}^{+})=d(v,\partial\mathcal{Y}^{+}\cap(M-\Bbb R_{+}^{p})).
\]
Thus,
\[
d_{H}(\partial\mathcal{P}^{k+}\cap(M-\Bbb R_{+}^{p}),\partial\mathcal{Y}^{+}\cap(M-\Bbb R_{+}^{p}))\leq\max_{v\in V^{k}}d(v,\mathcal{Y}^{\diamond}).
\]
From \eqref{eq:lem_ydiamond_convergence_1}, \eqref{eq:lem_ydiamond_convergence_2}
and the first equation of \eqref{eq:lem_ydiamond_convergence_0},
it follows that $\lim_{k\rightarrow\infty}{\rm WMin}\mathcal{P}^{k}={\rm WMin}\mathcal{Y}^{\diamond}$
which is the desired conclusion.\end{proof}

Consider the sets $\mathcal{Y}_{WN}$ and $V_{\varepsilon}$ obtained
from the algorithm. Recall that $\mathcal{L}=\mathcal{N}(\mathcal{Y}_{WN})$
and $\mathcal{U}=\mathcal{N}(V_{\varepsilon}).$ The following assertion
shows that these sets are the inner and outer approximate sets of
$\mathcal{Y}^{\diamond},$ respectively.

\begin{theorem}\label{theorem:yin_yout}

Let $\varepsilon=\epsilon e$, where $e$ denotes the vector of ones
in $\Bbb R^{p}$. We have the following properties:

i) $\mathcal{L}\subseteq\mathcal{Y}^{\diamond}\subseteq\mathcal{U}$;

ii) ${\rm WMin}\mathcal{Y}^{\diamond}\subseteq\mathcal{U}_{\varepsilon}\cap\mathcal{Y}^{\diamond}\subseteq\mathcal{Y}_{\varepsilon}^{\diamond}$;

iii) ${\rm WMin}\mathcal{L}\subseteq\mathcal{Y}_{\varepsilon}^{\diamond}$.

\end{theorem}

\begin{proof}i) is straightforward.

We now prove ii). When the algorithm terminates at the $K^{th}$ iteration,
with a tolerance level $\epsilon$, $\mathcal{U}\equiv\mathcal{P}^{K}$
is close enough to $\mathcal{Y}^{\diamond}$, namely 
\[
\mathcal{U}\subseteq\mathcal{Y}^{\diamond}+\epsilon U_{p},
\]
where $U_{p}$ is the closed unit ball in $\R^{p}$. Let $y\in{\rm WMin}\mathcal{Y}^{\diamond}$,
it is necessary to prove $y\in\mathcal{U}_{\varepsilon}$. For this
purpose, by definition of $\mathcal{U}_{\varepsilon}$, we need to
show that 
\begin{equation}
[(y-\epsilon e)-\mbox{int}\Bbb R_{+}^{p}]\cap\mathcal{U}=\emptyset.\label{eq:th_yin_yout_1}
\end{equation}
Indeed, since $\mathcal{U}\subseteq\mathcal{Y}^{\diamond}+\epsilon U_{p}\subseteq\mathcal{Y}^{\diamond}-\epsilon e+\Bbb R_{+}^{p}$,
it is sufficient to see that the two sets in the left hand side of
\eqref{eq:th_yin_yout_1} do not intersect. Moreover, if $y\in\mathcal{U}_{\varepsilon}\cap\mathcal{Y}^{\diamond}$,
obviously, \eqref{eq:th_yin_yout_1} is now true. Because $\mathcal{Y}^{\diamond}\subseteq\mathcal{P}^{K}$,
so $[(y-\epsilon e)-\text{int}\Bbb R_{+}^{p}]\cap\mathcal{Y}^{\diamond}=\emptyset$.
We conclude that $y\in\mathcal{Y}_{\varepsilon}^{\diamond}$.

The proof for iii) is similar.\end{proof}

\begin{theorem}\label{thm_ES} Given $\epsilon>0$ and $\varepsilon=\epsilon e$,
the algorithm terminates after a finite number of iterations and generates
the approximate solution set $ES$ such that $\mathcal{S}_{WE}\subseteq ES\subseteq\mathcal{S}_{\varepsilon}$.

\end{theorem}

\begin{proof}

Since $\epsilon>0$, from Lemma \ref{lem:lim_v_wv}, there exists
$K>0$ such that $||w_{v}-v||\leq\epsilon$ for all $v\in V^{K}$
at which iteration the algorithm terminates. Moreover, the necessary
and sufficient condition of a weakly efficient solution $x\in\mathcal{S}_{WE}$
is that there exists some $y\in{\rm WMin}\mathcal{Y}^{\diamond}$
such that $f(x)\leq y$. Theorem \ref{theorem:yin_yout}(ii) enables
us to write 
\[
\mathcal{S}_{WE}\subseteq\underset{y\in{\rm WMin}\mathcal{Y}^{\diamond}}{\bigcup}\{x\in\mathcal{S}\mid f(x)\leq y\}\subseteq\underset{y\in\mathcal{U}_{\varepsilon}\cap\mathcal{Y}^{\diamond}}{\bigcup}\{x\in\mathcal{S}\mid f(x)\leq y\}=ES.
\]
 Moreover, from the definition of $\mathcal{S}_{\varepsilon}$, we
have $ES\subseteq\mathcal{S}_{\varepsilon}$, so $\mathcal{S}_{WE}\subseteq ES\subseteq\mathcal{S}_{\varepsilon}$.
This completes the proof.\end{proof}

\section{Computational Experiment}

This section is used to illustrate our proposed algorithm through
several numerical examples. We also want to compare our results with
other related works of \cite{Benson2005}. The algorithms were implemented
in parallel on Intel(R) Xeon(R) CPU E5-2630 v4 at 2.20Ghz (32 logical
cores) and 128Gb RAM using Matlab 9.3 (2017b).

\begin{example} \label{example:linear_fractional}Consider the following
linear fractional programming problem

\[
\begin{array}{cccl}
\mbox{Vmin} & f_{1}(x) & = & \dfrac{-x_{1}}{x_{1}+x_{2}},\\
 & f_{2}(x) & = & \dfrac{3x_{1}-2x_{2}}{x_{1}-x_{2}+3}\\
\mbox{s.t.} & x_{1},x_{2}\in\Bbb R,\\
 & x_{1}-2x_{2} & \leq & 2,\\
 & -x_{1}-2x_{2} & \leq & -2,\\
 & -x_{1}+x_{2} & \leq & 1,\\
 & x_{1} & \leq & 6.
\end{array}
\]

\end{example}

Since both objectives of this problem are a ratio of two linear functions,
it is clearly a multiobjective strictly quasiconvex programming
problem. Thus, we can use our algorithm to solve this problem. Let us demonstrate below the detailed computation for the case $\epsilon=0.5$.

\textbf{Initialization step}. We find a lower boundary $m=(-1,-1)$ by solving the convex
programming problem \ref{eq:P_i}, $i=1,2$ and choose
an upper boundary $M=(0,2.6)$.

\textbf{Iteration} $k=0$. The only point in $V^0$ is chosen $v^0=m=(-1,-1)$. Solving $(P^2(v^0))$ gives $t_0=0.6404$ and $w^0=(-0.3596,-0.3596)$. Two new proper vertices $(-0.3596,-1.000)$ and $(-1.000,-0.3596)$ are inserted to $V^1$ by the cutting procedure. 

\textbf{Iteration} $k=1$. A point $v^1=(-0.3596,-1.000)$ is chosen from $V^1\setminus V_\varepsilon$. Solving $(P^2(v^1))$ gives $t_1=0.1626$ and $w^1=(-0.1970,-0.8374)$. Since $||w^1-v^1||=0.2299<\epsilon$, $V_\epsilon = V_\epsilon \cup \{v^1\}$. We continue to the next iteration.

\textbf{Iteration} $k=2$. Choose $v^2=(-1.0000,-0.3596)$. Solving $(P^2(v^2))$ gives $t_2=0.5257,w^2=(-0.4743,0.1661)$. Vertices $(-0.4743,-0.3596)$, $(-1.0000,0.1661)$ are inserted to $V^3$ by the cutting procedure.
   
\textbf{Iteration} $k=3$. Choose $v^3=(-0.4743,-0.3596)$. Solving $(P^2(v^3))$ gives $t_3=0.1005$ and $w^3=(-0.3738,-0.2591)$. Since $||w^3-v^3||=0.1421<\epsilon$, update $V_\epsilon = V_\epsilon \cup \{v^3\}$.

\textbf{Iteration} $k=4$. Choose $v^4=(-1.0000,0.1661)$. Solving $(P^2(v^4))$ gives $t_4=0.3571,w^4=(-0.6429,0.5232)$. Vertices $(-0.6429,0.1661)$, $(-1.0000,0.5232)$ are inserted to $V^5$ by the cutting procedure.

\textbf{Iteration} $k=5$. Choose $v^5=(-0.6429,0.1661)$.  Solving $(P^2(v^5))$ gives $t_5=0.1154,w^5=(-0.5275,0.2815)$. Since $||w^5-v^5||=0.1633<\epsilon$, update $V_\epsilon = V_\epsilon \cup \{v^5\}$.

\textbf{Iteration} $k=6$. Choose $v^6=(-1.0000,0.5232)$.  Solving $(P^2(v^6))$ gives $t_5=0.2377,w^5=(-0.7623,0.7608)$. Since $||w^6-v^6||=0.3361<\epsilon$, update $V_\epsilon = V_\epsilon \cup \{v^6\}$.

The algorithm terminates since there is no more points in $V^k\setminus V_\varepsilon$.  We obtain the sets
\begin{align*}
\mathcal{Y}_{WN} & = \{(-0.3596,-0.3596),(-0.1970,-0.8374),(-0.4743,0.1661),\\&(-0.3738,-0.2591),(-0.6429,0.5232),(-0.5275,0.2815),(-0.7623,0.7608)\}
\end{align*}
and
\[
V_\epsilon = \{(-0.3596,-1.0000),(-0.4743,-0.3596),(-0.6429,0.1661),(-1.0000,0.5232)\}.
\]

We compute the results in other
cases of $\epsilon$ and show them in Table \ref{table:example_1},
where T, V and C denote the average computation time, number of weakly
efficient values (or the number of vertices in $\mathcal{L}$) and
number of vertices of the outer approximate copolyblock, respectively.
The computational result is visualized in Figure \ref{fig:example_1_1}.

\begin{table}[h]
\begin{centering}
\begin{tabular}{lccc}
\hline 
\noalign{\smallskip}
$\epsilon$ & T & V & C\tabularnewline
\noalign{\smallskip}
\hline 
\hline 
0.1 & 1.894534 & 35 & 18\tabularnewline
\noalign{\smallskip}
\hline 
0.05 & 2.167836 & 67 & 34\tabularnewline
\noalign{\smallskip}
\hline 
0.025 & 2.644173 & 117 & 59\tabularnewline
\noalign{\smallskip}
\hline 
0.0125 & 4.124381 & 231 & 116\tabularnewline
\noalign{\smallskip}
\hline 
\noalign{\smallskip}
\end{tabular}
\par\end{centering}
\caption{\label{table:example_1}Computational results in Example \ref{example:linear_fractional}}
\end{table}

  \begin{figure}[h]
  \begin{centering}
  \includegraphics[scale=0.16]{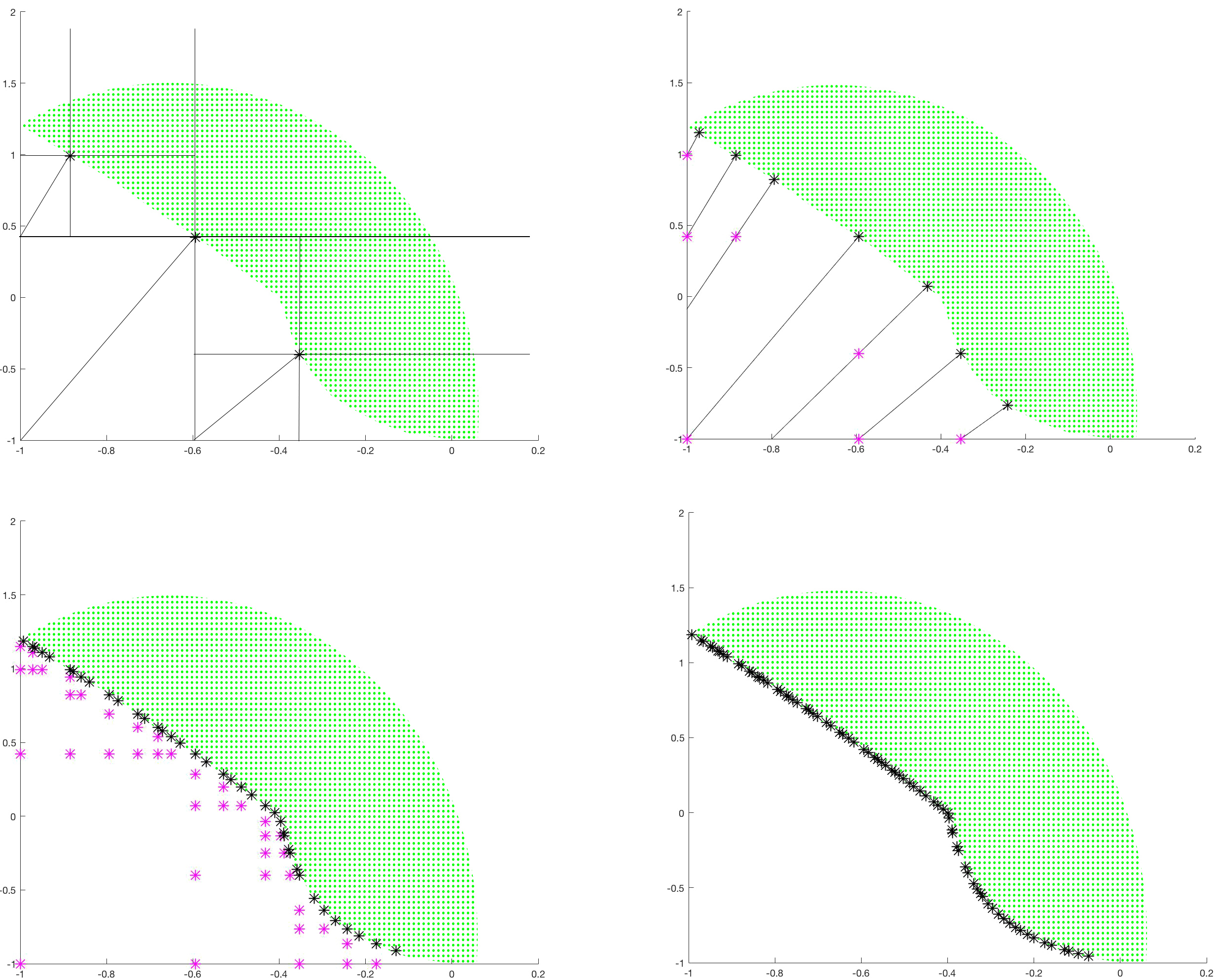}
  \par\end{centering}
  \noindent \centering{}\caption{\label{fig:example_1_1}The outer approximation $\mathcal{U}$ of
  $\mathcal{Y}^{\diamond}$ with different values of $\epsilon\in\{1,0.5,0.1,0.05\}$
  in Example \ref{example:linear_fractional}. The blue dots denote
  vertices of set $V^{k}$, while the red pluses represent the weakly
  nondominated points in the outcome space.}
  \end{figure}

\begin{example}\label{example:convex_fractional}Consider the following
convex fractional minimization problem

\[
\begin{array}{cccc}
\mbox{Vmin} & f_{1}(x) & = & \dfrac{x_{1}+1}{-x_{1}^{2}+3x_{1}-x_{2}^{2}+3x_{2}+3.50},\\
 & f_{2}(x) & = & \dfrac{x_{1}^{2}-2x_{1}+x_{2}^{2}-8x_{2}+20.00}{x_{2}}\\
\mbox{s.t.} & x_{1},x_{2}\in\Bbb R,\\
 & 2x_{1}+x_{2} & \leq & 6,\\
 & 3x_{1}+x_{2} & \leq & 8,\\
 & x_{1}-x_{2} & \leq & 1,\\
 & x_{1},x_{2} & \geq & 1.
\end{array}
\]
\end{example}

Since both $f_{1}(x)$ and $f_{2}(x)$ are convex fractions and the
feasible domain of this problem is a polyhedron, the above problem
is a strictly quasiconvex multiobjective programming problem.

By choosing $\epsilon=0.1,\hat{d}=(1,1)$, the algorithm stops after
19 iterations. We obtain the $V_{\varepsilon}$ set of 10 vertices
of the approximate copolyblock and the $\mathcal{Y}_{WN}$ set including
19 weakly nondominated points. The computational results with other
values of $\epsilon$ are presented in Table \ref{table:example_2_1}
with all notations having the same meaning as in Table \ref{table:example_1}.
The computational results are illustrated in Figure \ref{fig:example_2_1}.
Two green small circles in the figure are the nondominated outcome
points as stated in \cite{Benson2005} after running its algorithm
twice with different initial conditions.

\begin{table}[h]
\begin{centering}
\begin{tabular}{lccc}
\hline 
\noalign{\smallskip}
$\epsilon$ & T & V & C\tabularnewline
\noalign{\smallskip}
\hline 
\hline 
0.08 & 1.764705 & 25 & 13\tabularnewline
\noalign{\smallskip}
\hline 
0.04 & 2.156274 & 70 & 35\tabularnewline
\noalign{\smallskip}
\hline 
0.02 & 2.628011 & 110 & 55\tabularnewline
\noalign{\smallskip}
\hline 
0.01 & 3.529236 & 170 & 84\tabularnewline
\noalign{\smallskip}
\hline 
\smallskip
\end{tabular}
\par\end{centering}
\caption{\label{table:example_2_1}Computational results in Example \ref{example:convex_fractional} }
\end{table}

  \begin{figure}[h]
  \includegraphics[scale=0.5]{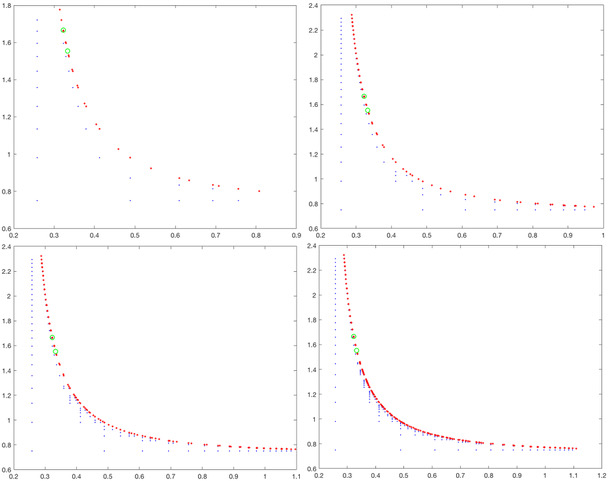}
  \centering{}\caption{\label{fig:example_2_1}The image of outer approximation $\mathcal{U}$
  of $\mathcal{Y}^{\diamond}$ with several values of $\epsilon=0.1,0.05,0.02,0.01$
  in Example \ref{example:convex_fractional}. The blue dots denote
  vertices of set $V^{k}$, while the red pluses represent the weakly
  nondominated points in the outcome space. Two green small circles
  are the nondominated outcome points as calculated by \cite{Benson2005}.}
  \end{figure}

\begin{example}\label{example:convex_3d}Consider the following convex
programming problem

\[
\begin{array}{cccl}
\mbox{Vmin} & f_{1}(x) & = & x_{1}^{2}+x_{2}^{2}+x_{3}^{2}+10x_{2}-120x_{3},\\
 & f_{2}(x) & = & x_{1}^{2}+x_{2}^{2}+x_{3}^{2}+80x_{1}-448x_{2}+80x_{3},\\
 & f_{3}(x) & = & x_{1}^{2}+x_{2}^{2}+x_{3}^{2}+448x_{1}+80x_{2}+80x_{3}
\end{array}
\]

\[
\begin{array}{crcl}
\mbox{s.t.} & x_{1}^{2}+x_{2}^{2}+x_{3}^{2} & \leq & 100,\\
 & 0 & \leq & x_{1}\leq10,\\
 & 0 & \leq & x_{2}\leq10,\\
 & 0 & \leq & x_{3}\leq10.
\end{array}
\]

\end{example}

\medskip

In the initial step, we obtain $m_{1}=(-1100,-4380,-4380)$ by solving
convex programming problems $(P_{1}^{m}),(P_{2}^{m}),(P_{3}^{m})$.
Instead of solving $(P_{1}^{M}),(P_{2}^{M}),(P_{3}^{M})$, we only
need to find a upper boundary $\hat{y}=(473.2051,1685.6406,1685.6406)$.
Let $m=(-1110,-4390,-4390)<m_{1},M=(2000,2000,2000)>\hat{y}$ and
a positive direction $\hat{d}=(0.2,1,1)>0$.

\medskip

Computational results with different values of $\epsilon$ are presented
in Table \ref{table:example_3_1} with the same notation meaning as
in Example \ref{example:linear_fractional}.

\begin{table}[h]
\begin{centering}
\begin{tabular}{lccc}
\hline 
\noalign{\smallskip}
$\epsilon$ & T & V & C\tabularnewline
\noalign{\smallskip}
\hline 
\hline 
800 & 4.748950 & 606 & 1,122\tabularnewline
\noalign{\smallskip}
\hline 
400 & 29.800921 & 13,450 & 24,982\tabularnewline
\noalign{\smallskip}
\hline 
200 & 4,766.430489 & 1,011,221 & 1,907,464\tabularnewline
\noalign{\smallskip}
\hline 
\noalign{\smallskip}
\end{tabular}
\par\end{centering}
\caption{\label{table:example_3_1}Computational results in Example \ref{example:convex_3d} }

\end{table}

  \begin{figure}[h]
  \begin{centering}
  \includegraphics[scale=0.5]{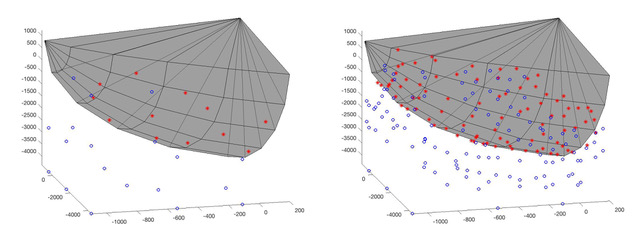}
  \par\end{centering}
  \caption{\label{fig:example_3_1}The distribution of the outer approximation
  $\mathcal{U}$ of $\mathcal{Y}^{\diamond}$ with $\epsilon=2400$
  and $400$. The blue circles denote vertices of set $V^{k}$, while
  the red asterisks represent the weakly nondominated points in the
  outcome space.}
  
  \end{figure}

\section{Conclusion}

In this paper, we propose an algorithm for solving the strictly quasiconvex
multiobjective programming problem \ref{QVP} based on the monotonic
approach. Firstly, we generate a weakly efficient solution of \ref{QVP}
associated with a nondominated outcome point by using strictly quasiconvex
programming. Then, we apply the cutting cones in monotonic optimization
to outer approximate the outcome set. From the sets of outer approximation
outcome points and nondominated outcome points, we have established
the inner and outer approximation set of outcome set and obtained
the approximation of weakly solution set that contains the whole weakly
solution set of \ref{QVP} with any tolerance. These properties are
guaranteed by the convergence theorems. We also develop a parallel
version of the proposed algorithm, that is more efficient than the
former. The numerical results show that the algorithm is flexible
for a large class of problems and the computational time is acceptable
for a feasible tolerance. In the future, the proposed algorithm can
be applied for solving many problems related to the multiobjective
programming problem \ref{QVP}.

\medskip
Received xxxx 20xx; revised xxxx 20xx.
\medskip

\end{document}